\documentclass{amsart}

\usepackage{amsmath,amssymb,amsfonts,amstext,amsthm}
\usepackage{url}
\usepackage{graphicx}
\input xy
\xyoption{all}
\usepackage[all]{xy}

\usepackage{subfiles}
\usepackage{tikz}

\newcommand{\nwc}{\newcommand}

\nwc{\aaa}{\mathcal{F}}
\nwc{\aap}{\mathcal{F}_{P}}
\nwc{\al}{\alpha}

\nwc{\C}{\mathbb{C}}
\nwc{\cb}{\overline{C}}
\nwc{\ccc}{\mathfrak{c}}
\nwc{\ch}{\widehat{C}}
\nwc{\cin}{\textbf{(v)}}
\nwc{\cl}{C'}
\nwc{\cp}{\mathcal{C}_{P}}
\nwc{\cpll}{\mathfrak{c}_{P'}}
\nwc{\ct}{\widetilde{C}}

\nwc{\dd}{\mathcal{L}}
\nwc{\ddd}{\mathfrak{d}}
\nwc{\ddl}{\mathcal{L}'}
\nwc{\dlp}{\delta_{P}}
\nwc{\doi}{\textbf{(ii)}}

\nwc{\enq}{$$}

\nwc{\fl}{\flushleft}
\nwc{\fff}{\mathcal{F}}
\nwc{\ffp}{\mathcal{F}_{P}}
\nwc{\ffq}{\mathcal{F}_{Q}}
\nwc{\ffl}{\mathcal{F}'}

\nwc{\G}{\mathcal{G}}
\nwc{\Ga}{\Gamma}
\nwc{\gtl}{\widetilde{g}}

\nwc{\hra}{\hookrightarrow}
\nwc{\hua}{h^{1}(C,\aaa )}

\nwc{\kk}{{\rm K}}

\nwc{\llb}{\mathcal{L}}

\nwc{\mb}{\mathbb}
\nwc{\mc}{\mathcal}
\nwc{\mm}{\mathfrak{m}}
\nwc{\mmp}{\mathfrak{m}_{P}}
\nwc{\mpd}{\mathfrak{m}_{P}^{2}}

\nwc{\nn}{\mathbb{N}}

\nwc{\ob}{\overline{\mathcal{O}}}
\nwc{\obr}{\mathcal{O}^*}
\nwc{\obp}{\overline{\mathcal{O}}_P}
\nwc{\och}{\mathcal{O}_{\hat{C}}}
\nwc{\oh}{\hat{\mathcal{O}}}
\nwc{\ohp}{\hat{\mathcal{O}}_{P}}
\nwc{\ol}{\mathcal{O}'}
\nwc{\oma}{\Omega (\mathfrak{a})}
\nwc{\omo}{\Omega (\mathcal{O})}
\nwc{\oo}{\mathcal{O}}
\nwc{\op}{\mathcal{O}_P}
\nwc{\opc}{\mathcal{O}_{P,C}}
\nwc{\oph}{\hat{\mathcal{O}}_{P}}
\nwc{\opl}{\mathcal{O}_{P}'}
\nwc{\oplc}{\mathcal{O}_{P,C}'}
\nwc{\opll}{\mathcal{O}_{P'}}
\nwc{\opt}{\tilde{\mathcal{O}}_{P}}
\nwc{\optt}{{\mathcal{O}}_{\tilde{P}}}
\nwc{\oq}{\mathcal{O}_{Q}}
\nwc{\oqt}{\tilde{\mathcal{O}}_{Q}}
\nwc{\ot}{\widetilde{\mathcal{O}}}
\nwc{\overop}{\bar{\oo}_{P}}

\nwc{\pb}{\overline{P}}
\nwc{\pbb}{P^*}
\nwc{\pbi}{\overline{P_{i}}}
\nwc{\pbr}{\overline{P_{r}}}
\nwc{\pgmd}{\mathbb{P}^{g+2}}
\nwc{\pgmu}{\mathbb{P}^{g+1}}
\nwc{\ph}{\hat{P}}
\nwc{\pp}{\mathbb{P}}
\nwc{\prv}{\noindent\textbook{Proof}:}
\nwc{\pt}{\widetilde{P}}
\nwc{\ptl}{\tilde{P}}
\nwc{\pum}{\mathbb{P}^{1}}

\nwc{\qh}{\hat{Q}}
\nwc{\qtl}{\tilde{Q}}
\nwc{\qua}{\textbf{(iv)}}

\nwc{\ra}{\rightarrow}
\nwc{\rh}{\hat{R}}

\nwc{\sei}{\textbf{(vi)}}
\nwc{\sep}{\beq\ast\ \ast\ \ast\enq}
\nwc{\sig}{\sigma}
\nwc{\Sig}{\Sigma}
\nwc{\ssp}{S_{P}}
\nwc{\sss}{{\rm S}}

\nwc{\tre}{\textbf{(iii)}}

\nwc{\um}{\textbf{(i)}}

\nwc{\vpb}{v_{\overline{P}}}
\nwc{\vtxp}{\widetilde{V}_{x,P}}
\nwc{\vxp}{V_{x,P}}

\nwc{\wh}{\hat{\omega}}
\nwc{\whp}{\hat{\omega}_{P}}
\nwc{\woch}{\omega\cdot\mathcal{O}_{\hat{C}}}
\nwc{\woh}{\omega\cdot\hat{\mathcal{O}}}
\nwc{\ww}{\omega}
\nwc{\wwb}{\omega^*}
\nwc{\wwct}{\omega _{\widetilde{C}}}
\nwc{\wwh}{\widehat{\omega}}
\nwc{\wwhp}{\widehat{\omega}_P}
\nwc{\wwp}{\omega _{P}}
\nwc{\wwt}{\widetilde{\omega}}
\nwc{\wwtp}{\widetilde{\omega}_P}

\nwc{\zz}{\mathbb{Z}}

\newtheorem{coro}{Corollary}[section]
\newtheorem{conv}[coro]{Convention}

\newtheorem{lemma}[coro]{Lemma}

\newtheorem{prop}[coro]{Proposition}

\newtheorem{rem}[coro]{Remark}
\newtheorem{rems}[coro]{Remarks}
\newtheorem{thm}[coro]{Theorem}
\newtheorem{conj}[coro]{Conjecture}

\let \fl=\flushleft

\let \ga=\gamma
\let \sub=\subset

\let \al=\alpha
\let \pr=\prime
\let \la=\lambda
\let \ka=\kappa

\begin{document}

\title{Singular rational curves with points of nearly-maximal weight}

\author{Ethan Cotterill}
\address{Instituto de Matem\'atica, UFF
Rua M\'ario Santos Braga, S/N,
24020-140 Niter\'oi RJ, Brazil}
\email{cotterill.ethan@gmail.com}

\author{Lia Feital}
\address{Departamento de Matem\'atica, CCE, UFV
Av. P H Rolfs s/n, 36570-000 Vi\c{c}osa MG, Brazil}
\email{liafeital@ufv.br}

\author{Renato Vidal Martins}
\address{Departamento de Matem\'atica, ICEx, UFMG
Av. Ant\^onio Carlos 6627,
30123-970 Belo Horizonte MG, Brazil}
\email{renato@mat.ufmg.br}




\maketitle

\begin{abstract}
In this article we study rational curves with a unique unibranch genus-$g$ singularity, which is of {\it $\ka$-hyperelliptic} type in the sense of \cite{To}; we focus on the cases $\ka=0$ and $\ka=1$, in which the semigroup associated to the singularity is of (sub)maximal weight. We obtain a partial classification of these curves according to the linear series they support, the scrolls on which they lie, and their gonality. 
\end{abstract}

\section*{Introduction}

Rational curves have long played an essential r\^ole in the classification of complex algebraic varieties. Even when the target variety is $\mb{P}^n$, the problem of classifying {\it singular} rational curves is surprisingly subtle, and most works to date have concentrated on the case $n=2$; see for example \cite{deB,FZ,KoP,Mo,O,Pi,Ton}. In this paper we explore the classification problem for rational curves with a unibranch singularity $P$ (and which are smooth away from $P$). 

\medskip
Semigroups of unibranch singularities naturally form a tree, whose vertices are indexed by their minimal generating sets. The asymptotic structure of the tree's infinite leaves is essentially prescribed by the {\it weights} of the underlying semigroups. This leads us naturally to a reconsideration of how these weights should be specified in the first place. In this line of questioning, we are guided by an obvious analogy between semigroups of singular points and semigroups of Weierstrass points of linear series on smooth curves.


\medskip
For a smooth curve, the Weierstrass semigroup is given either by pole orders of meromorphic functions or of vanishing orders of regular differentials, and via Serre duality each formulation is equivalent. For a singular curve, however, using pole orders or differential orders of vanishing produces two distinct notions of weight. The main premise of Section 1 of this paper is that using a notion of weight based on differential orders of vanishing, as in \cite{Ka} (who, however, makes no mention of semigroups), confers certain advantages.

\medskip
Torres \cite{To} has given an asymptotic classification of numerical semigroups according to the notion of weight coming from meromorphic functions. Our Theorem \ref{biell} (whose proof strongly uses Torres' results) gives a slightly more precise characterization, using a notion of weight based on differentials, for {\it bielliptic} singularities: these are shown to be precisely those nonhyperelliptic singularities of maximal weight. It is natural to speculate that Torres' general classification result for {\it $\ka$-hyperelliptic} singularities may be refined using our alternative version of weight. We make this quantitatively precise in Conjecture~\ref{W_K_conj}, and give some evidence for the conjecture in Theorem~\ref{kappa_hyperelliptic_weight}.

\medskip
The analogy with Weierstrass points also leads one naturally to wonder how the stratification of singular rational curves according to {\it gonality} interacts with the stratification according to value semigroups (indeed, one of the primary motivations behind Torres' theory was to clarify that causal relationship for smooth curves); this is the focus of Section 2. In Theorem~\ref{gonality_bound} we give a simple upper bound for the gonality of an integral projective curve as a function of the arithmetic genus, and characterize when the bound is sharp. 

\medskip
The study of (gonality) of linear series on irreducible singular curves is more delicate than the classical theory for smooth curves. It is both useful and necessary from our point of view to authorize {\it non-removable} base points, in the sense of \cite{Cp}. For example, it is well-known that any trigonal smooth curve lies on a surface scroll. St\"ohr and Rosa showed in \cite{SR} that the assertion remains true for {\it Gorenstein} curves, with the caveats that the scroll might be a cone; and that the $g^1_3$ might have non-removable base points. 
More generally, the fact that any smooth $d$-gonal curve embeds in a $(d-1)$-fold scroll is due to Bertini; see, e.g. \cite[Thm. 2.5]{Sc}. 
Theorem~\ref{lemgfb} establishes that rational curves with a single unibranch singularity lie on scrolls of a certain (co)dimension that is computable from their parametrizations. When the singularity is {\it bielliptic}, we also obtain (sufficient) conditions for the curve to admit pencils with a non-removable base point.  We also give an intrinsic characterization (in terms of $k$) of those rational curves that carry base-point-free $g^1_k$'s in Lemma~\ref{lemgk1}.


\medskip
Finally, in Section 3, we study rational curves with singularities of maximal and submaximal weight. Theorem \ref{thmhyp} characterizes hyperelliptic (singular) curves. It should be compared against the well-known characterization of hyperelliptic smooth curves as those for which 2 belongs to the Weierstrass semigroup in a point. Within our category of singular rational curves (with a single unibranch singularity) the analogous characterization fails: hyperelliptic singular curves have hyperelliptic singularities, i.e. singularities for which 2 belongs to the corresponding numerical semigroup, but not vice versa. Accordingly, we attempt to characterize those curves with hyperelliptic singularities that are not (globally) hyperelliptic, i.e. that admit a degree-2 morphism to $\mb{P}^1$. Proposition~\ref{hypgenus3} gives a complete resolution in the first nontrivial case of genus 3. We obtain some analogous results for bielliptic singular curves and curves with bielliptic singularities in Theorem~\ref{biell_curves} and Remark~\ref{biell_remark}, respectively.

\medskip
The connection between gonality, embeddings in scrolls, and rational curves with $\ka$-hyperelliptic singularities is a theme to which we intend to return in the future.

\subsection*{Acknowledgements} We would like to thank Maksym Fedorchuk, Joe Harris, Steve Kleiman, Karl-Otto St\"ohr, Fernando Torres, and Filippo Viviani for illuminating conversations, as well as the mathematics department at UFMG for making this collaboration possible. The first and third authors are partially supported by CNPq grant numbers 309211/2015-8 and 306914/2015-8, respectively. The second author is partially supported by FAPEMIG.

\subsection*{Conventions for rational curves and their singularities}
We work over $\mb{C}$. By {\it rational curve} we always mean a projective curve of geometric genus zero. A {\it numerical semigroup} is a subsemigroup ${\rm S} \sub \mb{N}$ of the natural numbers with finite complement $G_{\rm S}$; the {\it genus} $g=g({\rm S})$ is equal to the cardinality of $G_{\rm S}$.
The {\it genus} of a value semigroup of a singularity encodes the contribution of that singularity to the arithmetic genus of the underlying projective curve.








\section{Points of nearly-maximal weight}

\subsection{Weighting unibranch singularities}

A classical result of semigroup theory states that every numerical semigroup $\sss \sub \mb{N}$ with complement $G_{\sss}= \{\ell_1, \dots, \ell_g\}$ of cardinality $g$ is of {\it weight} 
\[
w(\sss):= \sum_{i=1}^g (\ell_i-i) \leq g(g-1)/2,
\]
with equality if and only if $2$  belongs to $\sss$. This in turn leads easily to a (well-known) fact that hyperelliptic curves are precisely those curves that admit points with semigroups of maximal weight. In an influential paper, T. Kato  \cite{Ka} succeeded in extending the classical story, obtaining an upper bound for the weight of a Weierstrass semigroup of an arbitrary point of a {\it nonhyperelliptic} curve, and classifying all maximal nonhyperelliptic curves. In this section we explore the extent to which an analogous story holds in the context of rational curves with a unique unibranch singular point.

\medskip
It is standard practice in algebraic geometry to let the {\it weight} of a nonsingular point $P$ of an integral and projective curve $C$ denote the weight of the semigroup 
\begin{equation}\label{standard_weight}
\sss_{C,P}=
\{n\in\nn\ |\,h^0(\oo_C((n-1)P))<h^0(\oo_C(nP))\}.
\end{equation}
On the other hand, when $P$ is a singularity, which we will assume throughout this section to be unibranch, there are other seemingly natural choices of weights available. Instead of using \eqref{standard_weight} we have opted to define the weight in terms of pole orders of differentials; it will turn out that this alternative version is obtained naturally as a perturbation of \eqref{standard_weight}, and  
that it agrees with \eqref{standard_weight} precisely when the singularity is Gorenstein.

\medskip
Concretely now, suppose $\pb$ is the preimage of a unibranch singularity $P$, and let 
\[
H^0(\ww_C)=\langle \lambda_1,\ldots,\lambda_g\rangle
\]
denote a basis for the space of global sections of the dualizing sheaf of the underlying (singular) curve $C$. Let $k_i:=|\vpb(\lambda_i)|$ denote the vanishing order of $\la_i$, where $\vpb$ is the valuation naturally associated to a desingularization of $P$. Now reorder the $\lambda_i$ so that 
\[
0<k_1<\ldots<k_{g-1}
\]
and accordingly set
\begin{equation}
\label{equdfp}
w(P):=\sum_{i=1}^{g-1}(k_i-i).
\end{equation}

\medskip
The latter definition warrants a bit of explanation. 
In \cite{Ka}, Kato defined the weight of a point on a (smooth) curve $C$, using global sections of the dualizing sheaf; classically, these were referred to as {\it differentials of first kind} (see, e.g., \cite{R}), and they are also called {\it regular differentials} in \cite{St2}. When $C$ is nonsingular, these sections are precisely those meromorphic differentials without poles. If $C$ is singular, they may admit poles along branches of singularities. From this point of view, an extremal situation is one in which all regular differentials of $C$ have poles along the branches of a singularity $P$. And indeed this happens precisely when $C$ is rational, $P$ is the unique singularity of $C$, and $P$ is unibranch, which is the case we are dealing with here. It is therefore quite natural to use pole orders (instead of the values themselves) of differentials when defining the weight of $P$; and in doing so, we extend Kato's original definition.

\medskip
For our purposes, it will also be useful to introduce a slight generalization of the usual semigroup-theoretic notion of weight based on pole orders. Namely, to {\it any} subset ${\rm T}$ of the natural numbers with cardinality-$g$ complement 
\begin{equation*}
\nn\setminus{\rm T}=\{\ell_1,\ell_2,\ldots,\ell_g\}
\end{equation*}
we let the {\it weight} $W_{\rm T}$ of ${\rm T}$ be the quantity
\begin{equation}
\label{equwei}
W_{\rm T}:= \sum_{i=1}^{g}\ell_i- \binom{g+1}{2}.
\end{equation}


\medskip
According to \cite[Thm. 2.8]{St2}, the weight may be explicitly computed from the semigroup $\sss$ of $P$. Indeed, let $c$ denote the conductor of $\sss$. Setting
\begin{equation}
\label{equdkk}
{\rm K}:=\{a\in\zz\,|\,c-a-1\not\in\sss\}.
\end{equation}
we have
\begin{equation}
\label{equwiw}
w(P)=W_{\rm K}
\end{equation}
where the latter is computed as in (\ref{equwei}) since $\kk$ may not be a semigroup in general. Rather, it is a semigroup if and and only if $\kk=\sss$, in which case $\sss$ is symmetric. This happens if and only if the point $P$ is Gorenstein, i.e., when $\ww_{C,P}$ is a free $\op$-module. 

\begin{rem} An interesting fact about (the various ways of computing) the weight of a subset ${\rm T}$ of $\nn$ with finite complement is hidden in \eqref{equwiw}. Namely, the right hand side of that equality, based on \eqref{equwei}, gives the standard version: one proceeds from 0 to the conductor, assigning to each gap encountered the number of positive integers in ${\rm T}$ left behind; these tallies of missed numbers are then summed together. The left hand side, based on \eqref{equdfp}, gives an alternative: one proceeds backward from the conductor to 0, assigning to each positive integer in ${\rm T}$ the number of gaps left behind. It is easy to check that both procedures yield the same value. For example, consider the set (not semigroup) ${\rm T}=\{0,1,3,4,6,7,9,10,12,\to\}$. Proceeding from 0 to the conductor 12, and assigning to each gap the numbers of positive elements of ${\rm T}$ left behind, we obtain:
\begin{equation}
\label{equfst}
W_{\rm T}=1+3+5+7=16.
\end{equation}
Similarly, if we proceed backward from 12 to 0, and assign to each positive integer in ${\rm T}$ the number of gaps left behind, we calculate:
\begin{equation}
\label{equscd}
W_{\rm T}=1+1+2+2+3+3+4=16.
\end{equation}

\medskip
Now consider the monomial curve $C:=\overline{(t^3,t^{13},t^{14})}\subset\mathbb{P}^3$; the bar notation, which follows \cite{Ab}, denotes projective closure. It is a rational curve of genus 8 with a unique singularity supported in $P=(0:0:0:1)$, which is unibranch. The semigroup of $P$ is $\sss=\{0,3,6,9,12,\to\}$, and hence $\kk={\rm T}$. So the weight of $P$ corresponds to \eqref{equfst} using the equality \eqref{equwiw}. On the other hand, one computes
$$
H^0(\ww_C)=\langle dt/t^2, dt/t^3, dt/t^5,dt/t^6,dt/t^8, dt/t^9,dt/t^{11},dt/t^{12}\rangle
$$
and from this point of view the weight of $P$ corresponds to \eqref{equscd} using definition \eqref{equdfp}.

\end{rem}

\subsection{The maximal case}

Next we characterize points of maximal weight.  As in the previous subsection (and throughout the rest of this section), a \emph{curve} $C$ is a rational integral and projective one-dimensional scheme of arithmetic genus $g$, with a unique singular point $P$, which is unibranch. 

\medskip
Adapting \cite{To}, we say that $P$ is of \emph{$\ka$-hyperelliptic type} if its semigroup $\sss$ satisfies the following properties: (a) $\sss$ has $\ka$ even numbers in $[2,4\ka]$; (b) $4\ka+2\in\sss$. For short, we say that $P$ is \emph{hyperelliptic} if it is $0$-hyperelliptic, i.e., $2\in\sss$.

\begin{thm}
\label{thmmwg}
The following are equivalent:
\begin{itemize}
\item[(i)] $w(P)$ is maximal as $(C,P)$ varies among curves of genus $g$;
\item[(ii)] $w(P)=\binom{g}{2}$;
\item[(iii)] $P$ is hyperelliptic.
\end{itemize}
\end{thm}

The proofs of both Theorem~\ref{thmmwg} and Theorem~\ref{biell}, which characterizes curves of {\it submaximal} weight, use the following auxiliary result, which may be of independent interest. 

\begin{lemma}\label{k_lemma}
Let $\sss$ be a value semigroup of genus $g \geq 0$, and let $\kk$ be as defined in \eqref{equdkk}. The corresponding weights are related by
\begin{equation}\label{K_versus_S}
W_{\kk}= W_{\sss}+ 2g-c.
\end{equation}
\end{lemma}

\begin{proof}[Proof of Lemma~\ref{k_lemma}.] We will give two proofs. The first is algebraic, the second is combinatorial. 

\medskip
{\it Algebraic proof.} It suffices to show that
\begin{equation}\label{K_versus_S_bis}
g+ g^{\pr}=c \text{ and } W_{\kk}+ g^{\pr}= W_{\sss}+ g
\end{equation}
where $g^{\pr}=g^{\prime}(\kk)$ denotes the {\it genus} of $\kk$, i.e. the cardinality of $\mb{N} \setminus \kk$. 

\medskip
To prove the first equality in \eqref{K_versus_S_bis}, let $\pi:\cb\to C$ be the normalization map, let $\overline{\oo}:=\pi_{*}(\oo_{\cb})$ and let $\mathcal{C}:=\mathcal{A}\text{nn}(\overline{\oo}/\mathcal{O})$ be the conductor sheaf of $C$. From \cite[Lem 1.6]{KM}, the dualizing sheaf $\omega$ of $C$ satisfies 
\[
\mc{O} \sub \omega \sub \overline{\mc{O}}.
\]
Since $\wwp$ is a canonical $\op$-module, from  \cite[Thm. 2.11]{St2} or \cite[Prop. 2.14.(iv)]{BDF} we have $v_{\pb}(\wwp)=\kk$ and from \cite[Lem. 19.(c)]{BF} we have that
$$
\ell(\obp/\op)=\ell(\op/\cp)+\ell(\wwp/\op).
$$ 
(The latter equality may also be deduced from \emph{local duality}, as in \cite[Thm. 1.5]{St2}.) It follows that
\begin{align*}
g+g' &=\#(\nn\setminus\sss)+\#(\nn\setminus\kk)\\
      &=\ell(\obp/\op)+\ell(\obp/\wwp)\\
      &=\ell(\op/\cp)+\ell(\wwp/\op)+\ell(\obp/\wwp)\\
      &=\ell(\obp/\cp) \\
      &=c.
\end{align*}

\medskip
To see why the second equality in \eqref{K_versus_S_bis} holds, note that the definition of $\kk$ implies that for every nonnegative integer $a \leq c-1$,
\[
a \in \kk \setminus \sss \iff c-1-a \in \kk \setminus \sss.
\]
Accordingly, set $m:= \#(\kk \setminus \sss)=g-g^{\pr}$. As usual, let $G_{\sss}= \{\ell_1,\dots,\ell_g\}$ denote the gap sequence of $\sss$, and let $\mb{N} \setminus \kk= \{\ell_{j_1}, \dots, \ell_{j_{g^{\pr}}}\} \sub G_{\sss}$ denote the gap sequence of $\kk$. Applying the definition of $W_{\sss}$, we obtain
\[
\begin{split}
W_{\sss}+g&= \sum_{i=1}^g \ell_i - \sum_{j=1}^{g-1} j \\
&= \sum_{i=1}^{g^{\pr}} \ell_{j_i}+ \frac{m(c-1)}{2}- \sum_{j=1}^{g-1} j \\
&= \sum_{i=1}^{g^{\pr}} \ell_{j_i}+ \frac{m(c-1)}{2}- \sum_{j=1}^{g^{\pr}-1} j - (g^{\pr}+ \cdots+ g-1) \\
&= \sum_{i=1}^{g^{\pr}} \ell_{j_i}+ \frac{m(c-1)}{2}- \sum_{j=1}^{g^{\pr}-1} j - \frac{m(g^{\pr}+g-1)}{2}
\end{split}
\]
Applying $g+g^{\pr}=c$, the right side of equality now becomes $W_{\kk}+g^{\pr}$, as desired.

\medskip
{\it Combinatorial proof.} Bras-Amor\'os and de Mier proved \cite{BdM} that each numerical semigroup may be represented as a Dyck path $\tau=\tau({\rm S})$ on a $g \times g$ square grid with axes labeled by $0,1, \dots, g$. Each path starts at $(0,0)$, ends at $(g,g)$, and has unit steps upward or to the right. Namely, the $i$th step of $\tau$ is up if $i \notin {\rm S}$, and is to the right otherwise. The weight $W_{\rm S}$ of ${\rm S}$ is then equal to the total number of boxes in the Young tableau $T_{{\rm S}}$ traced by the upper and left-hand borders of the grid and the Dyck path $\tau$. Indeed, the contribution of each gap $\ell$ of ${\rm S}$ to $W_{\rm S}$ is computed by the number of boxes inside the grid and to the left of the corresponding path edge.

\medskip
We may similarly encode $\kk=\kk({\rm S})$ as a path inside a $g_{\kk} \times g$ grid, where $g_{\kk}= \#\{\mb{N} \setminus \kk\}$; $W_{\kk}$ is then the total number of boxes in the Young tableau $T_{\kk}$ delimited by the path and the border of the grid. 

\medskip
We now focus on precisely how the tableaux associated to $W_S$ and $W_{\kk}$ differ. Namely, let $\ell_k \in G_{\rm S}$ denote the smallest gap larger than the smallest nonzero element in ${\rm S}$. The following facts are immediate consequences of construction.
\begin{itemize}
\item[i.] The total contribution of the gaps $\ell_j$, $k \leq j \leq g-1$ to $W_S$ is given by (the area of) a Young tableau $T_1$ which lies below the uppermost row of boxes in the bounding grid of ${\rm S}$.
\item[ii.] The boxes to the left of the subpath of $\kk$ with labels $c-1-\ell_k$, $k \leq j \leq g-1$ define a Young tableau that is transpose to $T_1$, and which lies below the uppermost row of boxes in the bounding grid of $\kk$. See Figures 1 and 2 for an illustration of the tableaux involved when ${\rm S}=\langle 4,10,11,17\}$, in which case $g=8$ and $c=14$.
\end{itemize}

\begin{figure}

\begin{tikzpicture}[scale=0.30]
\draw[blue, very thin] (0,0) rectangle (8,8);
\filldraw[draw=blue, fill=red] (0,3) rectangle (1,4);

\filldraw[draw=blue, fill=red] (0,4) rectangle (1,5);

\filldraw[draw=blue, fill=red] (0,5) rectangle (1,6);

\filldraw[draw=blue, fill=red] (0,6) rectangle (1,7);
\filldraw[draw=blue, fill=red] (1,6) rectangle (2,7);

\filldraw[draw=blue, fill=lightgray] (0,7) rectangle (1,8);
\filldraw[draw=blue, fill=lightgray] (1,7) rectangle (2,8);
\filldraw[draw=blue, fill=lightgray] (2,7) rectangle (3,8);
\filldraw[draw=blue, fill=lightgray] (3,7) rectangle (4,8);
\filldraw[draw=blue, fill=lightgray] (4,7) rectangle (5,8);

\end{tikzpicture}

\caption{Young tableau associated to ${\rm S}= \langle 4,10,11,17 \rangle$. The subtableau $T_1$ is indicated in red, while the contribution to $W_{\rm S}$ arising from the uppermost line is in grey.}
\end{figure}
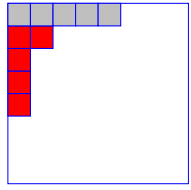


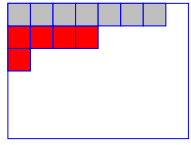
\begin{figure}

\begin{tikzpicture}[scale=0.30]
\draw[blue, very thin] (0,0) rectangle (8,6);
\filldraw[draw=blue, fill=red] (0,3) rectangle (1,4);

\filldraw[draw=blue, fill=red] (0,4) rectangle (1,5);
\filldraw[draw=blue, fill=red] (1,4) rectangle (2,5);
\filldraw[draw=blue, fill=red] (2,4) rectangle (3,5);
\filldraw[draw=blue, fill=red] (3,4) rectangle (4,5);

\filldraw[draw=blue, fill=lightgray] (0,5) rectangle (1,6);
\filldraw[draw=blue, fill=lightgray] (1,5) rectangle (2,6);
\filldraw[draw=blue, fill=lightgray] (2,5) rectangle (3,6);
\filldraw[draw=blue, fill=lightgray] (3,5) rectangle (4,6);
\filldraw[draw=blue, fill=lightgray] (4,5) rectangle (5,6);
\filldraw[draw=blue, fill=lightgray] (5,5) rectangle (6,6);
\filldraw[draw=blue, fill=lightgray] (6,5) rectangle (7,6);

\end{tikzpicture}

\caption{Young tableau associated to $K({\rm S})$ when ${\rm S}= \langle 4,10,11,17 \rangle$. The transpose of the tableau $T_1$ arising from the Young tableau of ${\rm S}$ is in red, while the contribution to $W_{\rm K}$ arising from the uppermost line is in grey.}
\end{figure}

\medskip
In particular, $W_{\kk}-W_{\rm S}$ is computed by the difference between the (number of boxes in) the uppermost rows of $T_{\kk}$ and $T_{\rm S}$, respectively. We conclude by noting that the uppermost row of $T_{\rm S}$ has $\ell_g-g=c-1-g$ boxes, while the uppermost row of $T_K$ has $g-1$ boxes.
\end{proof}

\begin{proof}[Proof of Theorem~\ref{thmmwg}.] First note that if $2\in\sss$ then $\sss$ is symmetric, so $w(P)=w(\sss)$, and $(iii) \implies (ii)$ follows trivially.

\medskip
\noindent $(ii) \implies (iii)$: From \eqref{equwiw}, we have $w(P)=W_{\kk}$. On the other hand, $W_{\kk}$ is defined by \eqref{equwei} and the dependency of $\kk$ on $\sss$ is spelled out by \eqref{equdkk}. We will derive the desired conclusion by showing that for any semigroup $\sss$ with $W_{\kk}=\binom{g}{2}$, we have $2\in\sss$. If $\sss$ is symmetric then $\kk=\sss$ and the result is classical. So it suffices to show that if $\sss$ is non-symmetric, we have $W_{\kk}<\binom{g}{2}$. For this purpose, we argue by induction on $g$. If $g=1$, the claim is vacuous. Now suppose the claim holds in all genera $g<\ga$, for some positive integer $\ga \geq 1$. Choose any non-symmetric semigroup $\sss^{\pr}$ of genus $\ga$. Note that \cite[Lem. 4]{BB} establishes that nonhyperelliptic symmetric semigroups are leaves of the semigroup tree, i.e. such semigroups have no descendents. In particular, the unique parent $\sss$ of $\sss^{\pr}$ in the semigroup tree is necessarily non-symmetric. By induction, it suffices to show that the corresponding weights $W_{\kk^{\pr}}$ and $W_{\kk}$ derived from $\sss^{\pr}$ and $\sss$, respectively, satisfy $W_{\kk^{\pr}}- W_{\kk} \leq \ga-1$.

\medskip
Applying \eqref{K_versus_S}, we see that the latter inequality, in turn, is equivalent to the assertion that
\begin{equation}\label{S_versus_Sprime_weights}
W_{\sss^{\pr}}- W_{\sss}+ 2+ c-c^{\pr} \leq \ga-1
\end{equation}
where $c$ and $c^{\pr}$ are the conductors of $\sss$ and $\sss^{\pr}$, respectively. However, it is also clear that
\[
c-1= \ell_{\ga-1}, \hspace{10pt} c^{\pr}-1= \ell_{\ga}, \text{ and } W_{\sss^{\pr}}- W_{\sss}= \ell_{\ga}
\]
where $\ell_{\ga}$ and $\ell_{\ga^{\pr}}$ are the largest gaps of $\sss$ and $\sss^{\pr}$, respectively. So \eqref{S_versus_Sprime_weights} is in fact equivalent to the assertion that
\begin{equation}\label{largestgap_S}
\ell_{\ga-1} \leq 2\ga-3.
\end{equation}
Finally, the inequality \eqref{largestgap_S} follows directly from the fact that $\sss$ is non-symmetric, since the conductor of a non-symmetric semigroup is always strictly less than twice the genus. 

\medskip
\noindent $(i) \iff (iii)$: In the previous item, we showed that $2 \notin \sss \implies W_{\kk}< \binom{g}{2}$. In particular, $W_{\kk}$ is not maximal whenever $2 \notin \sss$. The converse is obvious.
\end{proof}

\subsection{The submaximal case}

\medskip \noindent
In a similar vein, we say the singular point $P$ of $C$ is \emph{bielliptic} if it is $1$-hyperelliptic, i.e., $4$ and $6$ are the smallest positive integers in $\sss$. Our next result establishes that among singular points of a given genus, the bielliptic ones are precisely those of {\it submaximal} weight, i.e., of maximal weight among nonhyperelliptic points.


\begin{thm}\label{biell}
Let  $C$ be a curve of genus $g \geq 11$. Then the following are equivalent:
\begin{itemize}
\item[(i)] $w(P)$ is submaximal as $(C,P)$ varies among curves of genus $g$;
\item[(ii)] $w(P)=(g^2-5g+10)/2$;
\item[(iii)] $P$ is bielliptic.
\end{itemize}
\end{thm}

\begin{proof}[Proof of Theorem~\ref{biell}.] $(iii) \implies (ii)$: Assume that (iii) holds. There are two cases to consider, depending upon whether $\sss$ has one or two minimal generators besides 4 and 6. Set $\sss^*:=\{ s\in\sss\,|\, s\leq c\}$. In the first case, we have
\[
\sss^*=\{0,4,6,8,\dots,2g-4,2g-3,2g-2,2g\}
\]
and $\sss$ is symmetric, with minimal presentation $\langle 4,6,2g-3 \rangle$. We find that $W_{\kk}=W_{\sss}= (g^2-5g+10)/2$.

\medskip
In the second case, we have
\[
\sss^*=\{0,4,6,8,\dots,2g-2\}
\]
and $\sss$ has minimal presentation $\langle 4,6,2g-3,2g-1 \rangle$ (in particular, $\sss$ is nonsymmetric, with conductor $c=2g-2$).
We find that $W_{\sss}= (g^2-5g+6)/2$; $W_{\kk}= (g^2-5g+10)/2$ then follows from \eqref{K_versus_S}.

\medskip \noindent
$(ii) \implies (iii)$: It suffices to show that the minimal nonzero elements of any nonhyperelliptic semigroup $\sss$ with $W_{\kk}=(g^2-5g+10)/2$ are 4 and 6. For this purpose, we proceed by induction on $g$, much as in the proof of Thm~\ref{thmmwg}. Note first that the claim holds when $g=11$. Indeed, \cite[Lem. 3.4]{To2}, combined with \eqref{K_versus_S}, shows that the conductor $c$ of a non-hyperelliptic, non-bielliptic semigroup $\sss$ with at-least submaximal weight necessarily satisfies $c \leq 16$, or equivalently, that the largest gap $\ell_{11}$ of $\sss$ is necessarily such that $\ell_{11} \leq 15$. Now let $e_{\sss}$ denote the multiplicity of $\sss$. The fact that all eleven gaps of $\sss$ are at most 15 forces $e_{\sss} \geq 4$.

\medskip
Say, for the sake of argument that $e_{\sss}=4$. The next-smallest element in $\sss$ is then necessarily at least 13 (according to the requirement that all eleven gaps be at most 15). But then $\sss=\langle 4,13,18,19 \rangle$ is forced, and we obtain $W_{\kk}=23$, which is less-than-submaximal.

\medskip
Pushing this line of argument further, we find that the only numerical semigroups with $e_{\sss} \in \{5,6\}$ conforming to the requirement that all 11 gaps be at most 15 are the (unique) completions of the sets $\{5,11\}$, $\{5,12\}$, $\{5,14\}$, $\{6,8\}$, $\{6,9\}$, $\{6,10,13\}$, $\{6,10,14\}$, and $\{6,10,15\}$ to minimal generating sets of genus-11 semigroups. (In particular, the gap set of the semigroup corresponding to a given set $\{s_i\}_i$ consists of precisely those integers between 1 and 15 not realizable as a positive linear combination of the $s_i$.) We leave it to the reader to check that each of these semigroups has less-than-submaximal weight.

\medskip
On the other hand, note that
\begin{equation}\label{basic_ineq_W_S}
W_{\sss} \leq (11-(e_{\sss}-1))(15-(e_{\sss}-1)-e_{\sss})= (12-e_{\sss})(16-2e_{\sss})
\end{equation}
in general, with equality in \eqref{basic_ineq_W_S} holding if and only if $G_{\sss}= \{1,2,\dots,e_{\sss}-1\} \cup \{e_{\sss},\dots, 15\}$. In particular, \eqref{basic_ineq_W_S} implies that $W_{\sss} \leq 10$ whenever $e_{\sss} \geq 7$. It then follows immediately that $W_{\kk}= W_{\sss}+ 22-c$ is submaximal whenever $e_{\sss} \geq 7$.

\medskip
Now suppose the claim holds in all genera $11 \leq g<\ga$, for some positive integer $\ga \geq 12$. Then Torres' \cite[Thm. 3.6]{To2} guarantees the claim holds for all symmetric semigroups $\sss^{\pr}$ of genus $\ga$. Now say $\sss^{\pr}$ is non-symmetric. Just as in the proof of Thm~\ref{thmmwg}, it follows that the unique parent $\sss$ of $\sss^{\pr}$ in the semigroup tree is non-symmetric. Moreover, by induction using \eqref{K_versus_S}, we may conclude whenever the largest gap $\ell_{\ga-1}$ of $\sss$ satisfies
\begin{equation}\label{largest_gap_of_S}
\ell_{\ga-1} \leq 2\ga-5.
\end{equation}
In fact, \eqref{largest_gap_of_S} may fail, but only if $\ell_{\ga-1}= 2\ga-4$ and $c=2 \ga-3$. It follows by construction that the conductor $c^{\pr}$ of $\sss^{\pr}$ necessarily satisfies
\[
2\ga-2 \leq c^{\prime} \leq 2\ga-1.
\]
In particular, \eqref{K_versus_S} implies that 
\begin{equation}\label{W_K_ineq}
W_{\kk^{\pr}}\leq W_{\sss^{\pr}}+ 2.
\end{equation}
Now suppose that 4 and 6 are {\it not} the minimal elements in $\sss$. It then follows from loc. cit. that
\begin{equation}\label{torres}
W_{\sss^{\pr}}< (g^2-5g+6)/2;
\end{equation}
Combining \eqref{W_K_ineq} with \eqref{torres} yields $W_{\kk^{\pr}} < (g^2-5g+10)/2$, as desired.

\medskip \noindent $(i) \iff (iii)$: Obvious from the above, as in the proof of Theorem~\ref{thmmwg}.
\end{proof}

\begin{rem}
The hypothesis $g \geq 11$ is necessary (and optimal); for example, $\sss= \langle 3,11 \rangle$ is a non-bielliptic semigroup of genus 10 with $W_{\rm S}=W_{\kk}=30$, while $(g^2-5g+10)/2=30$.
\end{rem}

\subsection{Beyond the submaximal case.}

Theorems~\ref{thmmwg} and \ref{biell} were inspired by Torres' weight-based asymptotic characterization \cite[Thm. 3.6]{To2} of {\it $\ka$-hyperelliptic semigroups}.
Torres shows that whenever $g \gg \ka$, a numerical semigroup $\sss$ of genus $g$ is $\ka$-hyperelliptic if and only if the weight $W_{\sss}$ satisfies
\begin{equation}\label{kappa_weight_ineq}
{g-2\ka \choose 2}\leq W_{\sss} \leq {g-2\ka \choose 2}+2\ka^2.
\end{equation}
Our Theorem~\ref{biell} refines Torres' result when $\ka=1$, in the sense that our use of the modified weight $W_{\kk}$ in place of $W_{\sss}$ allows us to replace the inequality \eqref{kappa_weight_ineq} by an equality (in which the upper bound in \eqref{kappa_weight_ineq} remains). 

\medskip
It is natural to wonder whether an analogue of the characteristic inequality \eqref{kappa_weight_ineq} holds when $W_{\sss}$ is replaced by $W_{\kk}$. We speculate the following.

\begin{conj}\label{W_K_conj} Whenever $g \gg \ka$, $\sss$ is 
$\ka$-hyperelliptic if and only if the weight $W_{\kk}$ satisfies
\begin{equation}\label{kappa_weight_ineq_for_K}
{g-2\ka \choose 2}+ 2\ka \leq W_{\kk} \leq {g-2\ka \choose 2}+2\ka^2.
\end{equation}
\end{conj}

As evidence for Conjecture~\ref{W_K_conj}, we prove the following statement about the weight of a $\ka$-hyperelliptic semigroup $\sss$, under a convenient technical hypothesis.

\begin{thm}\label{kappa_hyperelliptic_weight}
Let $\sss$ denote a $\ka$-hyperelliptic semigroup of genus $g \gg \ka$, with next-to-largest gap $\ell_{g-1}$. Assume that each of the $\ka$ odd elements of $\sss \cap \{1,\dots,2g\}$ is strictly larger than $\ell_{g-1}$. Then \eqref{kappa_weight_ineq_for_K} holds.
\end{thm}

\begin{proof}[Proof of Theorem~\ref{kappa_hyperelliptic_weight}.] The fact that $\sss \cap \{1,\dots,2g\}$ has precisely $\ka$ odd elements is well-known, and indeed was crucially exploited by Torres in proving his \cite[Thm. 3.6]{To2}.
To prove our result, we will apply the (also well-known) facts that the smallest $\ka$ {\it even} elements $P_1< \dots <P_{\ka}$ of $\sss \cap \{1,\dots,2g\}$ satisfy $P_{\ka}=4\ka$, and that all even numbers greater than or equal to $4\ka$ belong to $\sss$.

\medskip
Given the combinatorial descriptions of $W_{\sss}$ and $W_{\kk}$ given in the proof of Lemma~\ref{k_lemma}, it now suffices to estimate the total contribution of the gaps $\ell_k, \dots, \ell_{g-1}$ to $W_{\sss}$, where $\ell_k$ is the smallest gap larger than the smallest nonzero element of $\sss$. In the terminology of the proof of Lemma~\ref{k_lemma}, this is precisely the total number of boxes in the Young tableau $T_1$.

\medskip
Under our technical hypothesis, however, the size of $T_1$ is controlled by the relative distribution of the even elements $P_1,\dots,P_{\ga}=4\ka$. In particular, it is clear that the number of boxes in $T_1$ is {\it minimized} when
\[
P_j= 2\ka+2j, j=1, \dots, \ka.
\]
With respect to these choices, $T_1$ is a staircase, i.e. each successive column has one fewer box than the preceding one. Since the first column of $T_1$ has precisely $(g-1)-(2\ka+2-1)=g-2\ka-2$ boxes, we deduce that its {\it weight} (i.e. total number of boxes) is given by
\[
W(T_1)= \binom{g-2\ka-1}{2}
\]
and it follows that
\[
W_{\kk}= W(T_1)+ g-1= {g-2\ka \choose 2}+ 2\ka.
\]

\medskip
Similarly, $W(T_1)$ is maximized when
\[
P_j= 4j, j= 1, \dots, \ka.
\]
With respect to these choices, $T_1$ is a tableau whose $j$th column is 3 boxes shorter than the $(j+1)$th, for all $j=1, \dots, \ka-1$, before stabilizing to a staircase from column $\ka$ onwards. It follows that
\[
\begin{split}
W(T_1)&= \sum_{j=1}^{\ka} (g-1-3j)+ \binom{g-1-3\ka}{2} \\
&= \ka(g-1)- 3\binom{\ka+1}{2}+ \binom{g-1-3\ka}{2} \\
&= \frac{g^2-(4\ka+3)g+ (6\ka^2+4\ka+2)}{2};
\end{split}
\]
and consequently, that
\[
W_{\kk}= W(T_1)+ g-1= \binom{g-2\ka}{2}+ \ka^2+ \ka.
\]

See Figure 3 for an illustration of tableaux associated to $\ka$-hyperelliptic semigroups that verify our technical hypothesis when $g=20$ and $\ka=3$.


\begin{figure}

\begin{tikzpicture}[scale=0.30]
\draw[blue, very thin] (0,0) rectangle (20,20);
\filldraw[draw=blue, fill=red] (0,3) rectangle (1,4);

\filldraw[draw=blue, fill=red] (0,4) rectangle (1,5);

\filldraw[draw=blue, fill=red] (0,5) rectangle (1,6);

\filldraw[draw=blue, fill=red] (0,6) rectangle (1,7);
\filldraw[draw=blue, fill=red] (1,6) rectangle (2,7);

\filldraw[draw=blue, fill=lightgray] (0,7) rectangle (1,8);
\filldraw[draw=blue, fill=red] (1,7) rectangle (2,8);

\filldraw[draw=blue, fill=lightgray] (0,8) rectangle (1,9);
\filldraw[draw=blue, fill=lightgray] (1,8) rectangle (2,9);

\filldraw[draw=blue, fill=lightgray] (0,9) rectangle (1,10);
\filldraw[draw=blue, fill=lightgray] (1,9) rectangle (2,10);
\filldraw[draw=blue, fill=lightgray] (2,9) rectangle (3,10);

\filldraw[draw=blue, fill=lightgray] (0,10) rectangle (1,11);
\filldraw[draw=blue, fill=lightgray] (1,10) rectangle (2,11);
\filldraw[draw=blue, fill=lightgray] (2,10) rectangle (3,11);
\filldraw[draw=blue, fill=lightgray] (3,10) rectangle (4,11);

\filldraw[draw=blue, fill=lightgray] (0,11) rectangle (1,12);
\filldraw[draw=blue, fill=lightgray] (1,11) rectangle (2,12);
\filldraw[draw=blue, fill=lightgray] (2,11) rectangle (3,12);
\filldraw[draw=blue, fill=lightgray] (3,11) rectangle (4,12);
\filldraw[draw=blue, fill=lightgray] (4,11) rectangle (5,12);

\filldraw[draw=blue, fill=lightgray] (0,12) rectangle (1,13);
\filldraw[draw=blue, fill=lightgray] (1,12) rectangle (2,13);
\filldraw[draw=blue, fill=lightgray] (2,12) rectangle (3,13);
\filldraw[draw=blue, fill=lightgray] (3,12) rectangle (4,13);
\filldraw[draw=blue, fill=lightgray] (4,12) rectangle (5,13);
\filldraw[draw=blue, fill=lightgray] (5,12) rectangle (6,13);

\filldraw[draw=blue, fill=lightgray] (0,13) rectangle (1,14);
\filldraw[draw=blue, fill=lightgray] (1,13) rectangle (2,14);
\filldraw[draw=blue, fill=lightgray] (2,13) rectangle (3,14);
\filldraw[draw=blue, fill=lightgray] (3,13) rectangle (4,14);
\filldraw[draw=blue, fill=lightgray] (4,13) rectangle (5,14);
\filldraw[draw=blue, fill=lightgray] (5,13) rectangle (6,14);
\filldraw[draw=blue, fill=lightgray] (6,13) rectangle (7,14);

\filldraw[draw=blue, fill=lightgray] (0,14) rectangle (1,15);
\filldraw[draw=blue, fill=lightgray] (1,14) rectangle (2,15);
\filldraw[draw=blue, fill=lightgray] (2,14) rectangle (3,15);
\filldraw[draw=blue, fill=lightgray] (3,14) rectangle (4,15);
\filldraw[draw=blue, fill=lightgray] (4,14) rectangle (5,15);
\filldraw[draw=blue, fill=lightgray] (5,14) rectangle (6,15);
\filldraw[draw=blue, fill=lightgray] (6,14) rectangle (7,15);
\filldraw[draw=blue, fill=lightgray] (7,14) rectangle (8,15);

\filldraw[draw=blue, fill=lightgray] (0,15) rectangle (1,16);
\filldraw[draw=blue, fill=lightgray] (1,15) rectangle (2,16);
\filldraw[draw=blue, fill=lightgray] (2,15) rectangle (3,16);
\filldraw[draw=blue, fill=lightgray] (3,15) rectangle (4,16);
\filldraw[draw=blue, fill=lightgray] (4,15) rectangle (5,16);
\filldraw[draw=blue, fill=lightgray] (5,15) rectangle (6,16);
\filldraw[draw=blue, fill=lightgray] (6,15) rectangle (7,16);
\filldraw[draw=blue, fill=lightgray] (7,15) rectangle (8,16);
\filldraw[draw=blue, fill=lightgray] (8,15) rectangle (9,16);

\filldraw[draw=blue, fill=lightgray] (0,16) rectangle (1,17);
\filldraw[draw=blue, fill=lightgray] (1,16) rectangle (2,17);
\filldraw[draw=blue, fill=lightgray] (2,16) rectangle (3,17);
\filldraw[draw=blue, fill=lightgray] (3,16) rectangle (4,17);
\filldraw[draw=blue, fill=lightgray] (4,16) rectangle (5,17);
\filldraw[draw=blue, fill=lightgray] (5,16) rectangle (6,17);
\filldraw[draw=blue, fill=lightgray] (6,16) rectangle (7,17);
\filldraw[draw=blue, fill=lightgray] (7,16) rectangle (8,17);
\filldraw[draw=blue, fill=lightgray] (8,16) rectangle (9,17);
\filldraw[draw=blue, fill=lightgray] (9,16) rectangle (10,17);

\filldraw[draw=blue, fill=lightgray] (0,17) rectangle (1,18);
\filldraw[draw=blue, fill=lightgray] (1,17) rectangle (2,18);
\filldraw[draw=blue, fill=lightgray] (2,17) rectangle (3,18);
\filldraw[draw=blue, fill=lightgray] (3,17) rectangle (4,18);
\filldraw[draw=blue, fill=lightgray] (4,17) rectangle (5,18);
\filldraw[draw=blue, fill=lightgray] (5,17) rectangle (6,18);
\filldraw[draw=blue, fill=lightgray] (6,17) rectangle (7,18);
\filldraw[draw=blue, fill=lightgray] (7,17) rectangle (8,18);
\filldraw[draw=blue, fill=lightgray] (8,17) rectangle (9,18);
\filldraw[draw=blue, fill=lightgray] (9,17) rectangle (10,18);
\filldraw[draw=blue, fill=lightgray] (10,17) rectangle (11,18);

\filldraw[draw=blue, fill=lightgray] (0,18) rectangle (1,19);
\filldraw[draw=blue, fill=lightgray] (1,18) rectangle (2,19);
\filldraw[draw=blue, fill=lightgray] (2,18) rectangle (3,19);
\filldraw[draw=blue, fill=lightgray] (3,18) rectangle (4,19);
\filldraw[draw=blue, fill=lightgray] (4,18) rectangle (5,19);
\filldraw[draw=blue, fill=lightgray] (5,18) rectangle (6,19);
\filldraw[draw=blue, fill=lightgray] (6,18) rectangle (7,19);
\filldraw[draw=blue, fill=lightgray] (7,18) rectangle (8,19);
\filldraw[draw=blue, fill=lightgray] (8,18) rectangle (9,19);
\filldraw[draw=blue, fill=lightgray] (9,18) rectangle (10,19);
\filldraw[draw=blue, fill=lightgray] (10,18) rectangle (11,19);
\filldraw[draw=blue, fill=lightgray] (11,18) rectangle (12,19);

\end{tikzpicture}

\caption{Tableaux $T_1$ associated with weight-minimizing and weight-maximizing $\ka$-hyperelliptic semigroups that verify our technical hypothesis when $g=20$ and $\ka=3$. The (irrelevant) uppermost line is left empty, the minimal $T_1$ is in grey, and the disparity $\ka^2-\ka=6$ between minimal and maximal weight is in red.}
\end{figure}
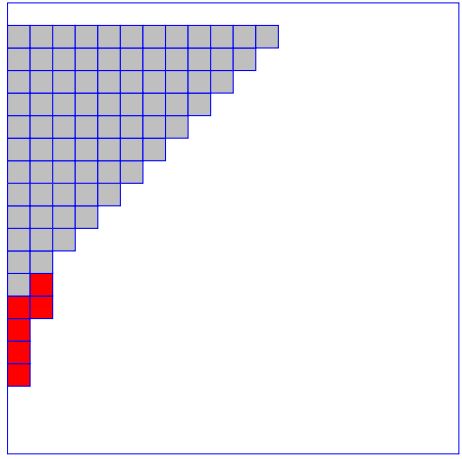

\end{proof}

\begin{rems}
Clearly the technical hypothesis posited in Theorem~\ref{kappa_hyperelliptic_weight} does not hold in general; indeed, Torres proved in \cite[Lem 3.2]{To2} that $W_{\sss}$ achieves the maximum value $\binom{g-2\ka}{2}+ 2\ka^2$ possible if and only if
\[
\sss= \sss_0:= \langle 4, 4\ka+2, 2g-4\ka+1 \rangle,
\]
which violates our technical hypothesis. Indeed, when $\sss= \sss_0$, we can only assert that all odd elements of $\sss \cap \{1,\dots,2g\}$ are strictly greater than the $(g-\ka)$th gap $\ell_{g-\ka} \in G_{\sss}$. Note that the semigroup $\sss_0$ is symmetric, so $W_{\sss}=W_K$ in this case. (In the statement of \cite[Lem 3.2]{To2} there is a misprint; the $4\ka$ that appears there should be replaced by $4\ka+2$.)

\medskip
The significance of our technical hypothesis is that it ensures that $T_1$ stabilizes to a staircase, while $\sss_0$ does not exhibit this behavior-- indeed, its corresponding tableau $T_1$ is symmetric. More to the point, we expect $W_{\kk}$ to be maximized precisely when $\sss=\sss_0$. Nevertheless, the argument we have used to prove Theorem~\ref{kappa_hyperelliptic_weight} is consistent with the general principle that the weight is maximized when the $\ka$ smallest nonzero even elements of $\sss$ are precisely $4j, j=1, \dots, \ka$. So it seems reasonable to expect that an appropriate generalization of the argument we have used will prove Conjecture~\ref{W_K_conj}.
\end{rems}

\section{Linear Series, Gonality and Scrolls}


Having fixed a choice of definition for the weight of a unibranch singularity, we now take a moment to review the theory of linear series on (irreducible) singular curves. For us, the key point is that linear series on singular curves may admit ``non-removable" base points, in a sense made precise by Coppens \cite{Cp}. Linear series with non-removable base points play a key r\^ole in Rosa and St\"ohr's geometric classification of trigonal Gorenstein curves \cite{SR}, which naturally generalizes the well-known classification of smooth trigonal curves. They also serve to describe pencils on non-Gorenstein curves; see \cite[Ex. 2.2]{AM}.


\subsection{Conventions for linear series on singular curves}
We use $g_{d}^{r}$ as a shorthand for a linear series of degree $d$ and dimension $r$. Here some care is required, as our curves are singular. For our purposes, a \emph{linear series of dimension $r$ in $C$} will denote any set of fractional ideals of the form
$$
\dd =\dd(\aaa ,V):=\{x^{-1}\aaa\ |\ x\in V\setminus 0\}
$$
where $\mathcal{F}\subset\mathcal{K}_C$ is a torsion-free rank one subsheaf of the canonical sheaf of $C$, and $V$ is a vector subspace of $H^{0}(\aaa )$ of dimension $r+1$. The \emph{degree} of $\dd$ is
$$
k:=\deg \aaa :=\chi (\aaa )-\chi (\oo_C)
$$
where $\chi$ is the holomorphic Euler characteristic. In particular, when $\oo_C\subset\aaa$ we have $\deg\aaa=\sum_{P\in C}\dim(\aaa_P/\op)$.

A point $P\in C$ is called a \emph{base point of} a linear series $\dd=\dd(\aaa ,V)$ if $x\op\subsetneq\aaa_P$ for every $x\in V$. A base point is called \emph{removable} if $\dd(\oo\langle V\rangle,V)$ has no base points, where $\oo\langle V\rangle$ is the sheaf generated by the sections in $V$ (this notion is due to Coppens; see \cite{Cp}). So $P$ is a non-removable base point of $\dd$ if and only if the stalk $\oo_C\langle V\rangle_P$ is not a free $\op$-module; if so, $P$ is necessarily singular. 

\subsection{Gonality of singular curves}
An integral and projective curve $C$ is said to be \emph{hyperelliptic} whenever a degree-$2$ morphism $C\to\pum$ exists. Note that this definition applies irrespective of whether $C$ is nonsingular or not. 
However, the latter definition is {\it not}, in the general case, equivalent to stipulating that the curve carry a $g_2^1$. Rather, as will be spelled out in Theorem \ref{thmhyp}, for the usual equivalence between these two characterizations to hold, we should insist that the $g_2^1$ be base-point-free.
Curves carrying a $g_2^1$ with a non-removable base point were characterized in \cite[Thms. 3.4, 5.10]{KM} as rational nearly normal. More generally, we have the following.

\begin{lemma}
\label{lemmsl}
There exists a morphism $C\to\mathbb{P}^1$ of degree $k$ if and only if $C$ carries a base-point-free $g_k^1$.
\end{lemma}

\begin{proof}
There exists a morphism $\phi:=(x_0,x_1):C\to\pum$ if and only if for each $Q\in C$ there is some $i(Q)\in\{0,1\}$ such that $x_i/x_{i(Q)}\in\oo_Q$ for $i=0,1$. Equivalently, $x_0\oo_Q+x_1\oo_Q=x_{i(Q)}\oo_Q$ for each $Q\in C$. That is, $\aaa:=\oo_C\langle x_0,x_1\rangle$ is invertible, which is the same as saying that $\mathcal{L}(\aaa,\langle x_0,x_1\rangle)$ is base-point-free. Since $\deg(\phi)=\deg(\aaa)$, the equivalence follows.
\end{proof}

We call the \emph{gonality} of $C$, and denote by ${\rm gon}(C)$, the smallest $k$ for which $C$ carries a $g_k^1$.

\begin{thm}\label{gonality_bound}
Let $C$ be an integral and projectice curve of genus $g$ over an algebraically closed field. Then 
$$
{\rm gon}(C)\leq g+1
$$
and if the bound is attained then $C$ is Gorenstein, and its normalization is either $\mathbb{P}^1$, i.e., $C$ is rational, or an elliptic curve.
\end{thm}

\begin{proof}
Follow the proof of \cite[Thm. 3.(i)]{CFM} and use the fact that $C$ is Gorenstein if and only if all of the local quotients $\omega_P/\mc{O}_P$ (as $P$ varies over points in $C$) are zero.
\end{proof}

For later use, we recall the following definition/notation: the $r$-fold scroll $S=S_{m_1m_2\ldots m_r}$ is the set of points $S=\{(x_0:\ldots: x_n)\}\subset\mathbb{P}^n$ such that the rank of the matrix
$$
\left[
      \begin{array}{cccc}
        x_0\, \ldots\, x_{m_1-1} &  x_{m_1+1}\,\ldots \, x_{m_1+m_2} & \ldots & x_{m_1+\ldots+m_{r-1}+r-1}\, \ldots\, x_{n-1}\\
        x_1\,   \ldots\,  x_{m_1}  &  \ \ \ x_{m_1+2}\, \ldots\,  x_{m_1+m_2+1}&\ldots & x_{m_1+\ldots+m_{r-1}+r}\,\ldots\,x_n
\end{array}\right]
$$
is less than 2.

\

For the remainder of this article, a \emph{curve} $C$ will be again a rational curve with a unique singular point $P$ which is unibranch. Let $\pb\in \cb=\pum$ denote the preimage of $P$ under the normalization map. The semigroup of $P$ is then $\sss=\vpb(\op)$. 
Let
$$
m:=m_C(P)
$$
denote the multiplicity of $C$ at $P$. 

\begin{conv}
\label{cnvcnv}
\emph{Write $K(C)=\C(t)$ and $\pum=\C\cup\{\infty\}$ so that $t$ is the identity function at finite distance, and $\pb=0$. In particular, $t$ is a local parameter at $\pb$. Then write $C=\{P\}\cup(\pum\setminus \{0\})$, i.e., we identify the regular points of $C$ with their pre-images under the normalization map. }
\end{conv}

With this convention, we obtain the following characterization of curves (in the sense of this section) equipped with base-point-free pencils.



\begin{lemma}
\label{lemgk1}
$C$ carries a base point free $g_k^1$ if and only if there exist $f,h\in\C[t]$ with no common factor and $f(0)=0$, such that $f/h\in\op$, and
$$
k=\deg(h)+{\rm max}\{0,\deg(f)-\deg(h)\}
$$
\end{lemma}

\begin{proof}
Any base point free $g_k^1$ in $C$ can be given by $\mathcal{L}(\oo_C\langle 1,f/h\rangle,\langle 1,f/h\rangle)$ where $f,h\in\C[t]$, $f/h\in\oo_P$, and we may take $f$ and $h$ with no common factors and, subtracting a suitable constant, we may further assume that $f(0)=0$. Setting $\aaa:=\oo_C\langle 1,f/h\rangle$ we have that the degree of $\mathcal{L}$ is $k=\deg(\aaa)$. Now
$$
\deg(\aaa)=\sum_{Q\in C}\ell({\aaa}_Q/\oo_Q)
$$
where
$$
\ell({\aaa}_Q/\oo_Q)=\dim_\C ((\oo_Q+f/h\cdot\oo_Q)/\oo_Q).
$$
Write $h=(t-c_1)^{m_1}\ldots(t-c_l)^{m_l}$ with $c_i\in\C$. Then
$$
\ell({\aaa}_Q/\op)=
\begin{cases}
0 & \text{if}\ Q=P \\
m_i & \text{if}\ Q=c_i \\
0 & \text{if}\ Q\in \C\setminus\{0,c_1,\ldots,c_l\} \\
{\rm max}\{0,\deg(f)-\deg(h)\} & \text{if}\ Q=\infty
\end{cases}
$$
The result follows since $\sum_{i=1}^l m_i= \deg(h)$.
\end{proof}

\begin{thm}
\label{lemgfb}
Let $C$ denote a curve (subject to the standing hypotheses of this section) that is the image of a morphism $f=(f_0,\dots,f_n): \mb{P}^1 \ra \mathbb{P}^n$, with $f_i\in\C[t]$ for all $i$. Normalize $f$ by letting
\begin{align*}
f_0&=1+a_rt^r+a_{r+1}t^{r+1}+\ldots, \text{ and }\\
f_1 &=t^m+b_{m+s}t^{m+s}+b_{m+s+1}t^{m+s+1}+\ldots.
\end{align*}
\begin{itemize}
\item[(i)] Assume $m=2$ and that one of the following conditions holds:
\begin{itemize}
\item[] $r$ is odd, and either $r<s$ or $s=0$;
\item[] $s$ is odd, and either $s<r$ or $r=0$; 
\item[] $r=s$, both are odd, and $a_r\neq a_{m+s}$.
\end{itemize}
Then $C$ carries a $g_k^1$ with a non-removable base point where
$$
k=2+g-\frac{{\rm min}\{r,s\}+1}{2}.
$$
\item[(ii)] Write $f_0=(t-c_1)^{m_1}\ldots(t-c_l)^{m_l}$ and, 
for $f\in V:=\langle f_1,\ldots,f_n\rangle\subset \mathbb{C}[t]$, write 
$$
f=(t-c_1)^{m_{1,f}}\ldots(t-c_l)^{m_{l,f}}h_f(t)
$$
with $h_f(c_j)\neq 0$ for $1\leq j\leq l$. Set
$m'_j:={\rm min}\{m_{j,i}\}_{i=1}^n$, $d':={\rm max}\{\deg(f_i)\}_{i=1}^n$,
and define the vector space
$$
U:=\bigg\{ f\in V\ 
\bigg{|}
\begin{array}{cl}
\deg(f)\leq d'+\deg(f_1)-\deg(f_0) &\text{if} \ \deg(f_0)<\min(\deg(f_1),\deg(f))  \\ 
 m_{j,f}\geq m_j'+m_j-m_{j,i} &\text{if} \ m_j>\max(m_i,m_{j,f})
\end{array}
\ \
 \bigg\}.
$$
Then $C$ lies on a scroll $S\subset\mathbb{P}^n$ of codimension $\dim(U)$.
\end{itemize}
\end{thm}

\begin{proof}
To prove (i), consider the power series 
\begin{equation}
\label{equlau}
x:=f_{1}/f_0=
\begin{cases}
t^2-a_rt^{r+2}+\ldots & \text{if}\ r<s\\
t^2+b_{2+s}t^{2+s}+\ldots  & \text{if}\ s<r \\
t^2+(b_{2+s}-a_r)t^{r}+\ldots  & \text{if }r=s
\end{cases}
\end{equation}
and let $\aaa:=\oo_C\langle 1,t^2\rangle$. We claim the linear series $\mathcal{L}(\aaa,\langle 1,t^2\rangle)$ is a $g_k^1$ which has $P$ as a non-removable base point. To see this, let $\pb$ be the preimage of $P$ under the normalization $\cb\to C$ and set $a:=2+{\rm min}\{r,s\}$. We have $y:=x-ct^2\in\aaa_P$ for any $c\in k$ and, by (\ref{equlau}), we may choose $c$ such that $y=bt^a+\ldots$ for some nonzero $b\in k$; that is, such that $a\in \vpb(\aaa_P)$. But $\aaa_P$ is an $\op$-module, so $x^i y\in \aaa_P$ for any $i\geq 0$, which implies that $a+2i\in\vpb(\aaa_P)$ for any $i\geq 0$. Therefore we have
\begin{equation}
\label{equvf1}
\vpb(\aaa_P)=\{0,2,4,\ldots,a-1,\to\}.
\end{equation}
From \eqref{equvf1}, we see there is no $q\in\mathbb{Z}$ for which $\vpb(\aaa_P)=\sss+q$ and thus, there is no function $z\in k(C)$ for which $\aaa_P=z\cdot\op$. That is, $\aaa_P$ is not free, i.e., $P$ is a non-removable base point of $\mathcal{L}$. 

\medskip
In order to compute the degree of $\mathcal{L}$ we express the set $\vpb(\aaa_P)$ in a different way. Namely, setting $e:=g-(a+1)/2$ we have
$$
\vpb(\aaa_P)=\sss\cup\{a,a+2,a+4,\ldots,a+2e\}.
$$
On the other hand,
$$
\ell(\aaa_P/\op)=\#(\vpb(\aaa_P)\setminus\sss)=e+1
$$
and we have
$$
\ell(\aaa_Q/\op)=
\begin{cases}
e+1 & \text{if}\ Q=P \\
0 & \text{if}\ Q\in k\setminus\{0\} \\
2 & \text{if}\ Q=\infty
\end{cases}
$$
so the item follows.

\medskip
To prove (ii), we adapt an argument of Schreyer's \cite[(2.2), p.113]{Sc}. Consider the very ample sheaf $\oo_C(1)=\oo_C\langle 1,x_1,\ldots,x_n\rangle$, where $x_i:=f_i/f_0$. Since $\oo_C(1)_P=\op$ and since we identify smooth points of $C$ with their preimages, we may write $\oo_C(1)=\oo_C(H)$ where $H$ is the Weil divisor on $\pum$ given by
\[
H=\left(\sum_{j=1}^l {\rm max}\{0,m_j-m'\}\cdot c_j\right)+{\rm max}\{0,d'-\deg(f_0)\}\cdot\infty.
\]
Now consider the sheaf $\aaa:=\oo_C\langle 1,f_0/f_1\rangle$, which defines a $g_k^1$ on $C$ as in Lemma~\ref{lemgk1}. We may write $\aaa:=\oo_C(D)$ where
\[
D:=\left(\sum_{j=1}^l {\rm max}\{0,m_j-m_{1,j}\}\cdot c_j\right)+{\rm max}\{0,\deg(f_1)-\deg(f_0)\}\cdot\infty.
\]
Finally, set $G:=\langle 1,f_1/f_0\rangle\subset H^0(\oo_C(D))$, and consider the multiplication map
\begin{equation}\label{multiplication_map}
G\otimes H^0(\oo_C(H-D))\longrightarrow H^0(\oo_C(H)).
\end{equation}

The map \eqref{multiplication_map} is given by a $2\times h^0(\oo_C(H-D))$ matrix whose minors cut out a scroll $S\subset\mathbb{P}^n$ for which
\[
{\rm codim}(S)=h^0(\oo_C(H-D))-1.
\]
Now $\C\subset H^0(\oo_C(H-D))$ and it is not hard to see that 
\[
U=\{f\in V\ |\ f/f_0\in H^0(\oo_C(H-D))\}.
\]
Accordingly, we have
\[
H^0(\oo_C(H-D))=\C \oplus (U/f_0)
\]
and the claim follows.
\end{proof}

\section{Curves with points of nearly-maximal weight}

\subsection{The maximal case}

\noindent Given the classical equivalence between hyperelliptic (smooth) curves and curves with hyperelliptic Weierstrass semigroups, it is natural to wonder whether this characterization extends to our setting as well. This is not true in general, as we shall see. We begin by characterizing (globally) hyperelliptic singular rational curves.

\begin{thm}
\label{thmhyp}
Let  $C$ be a curve of genus $g$. Then the following are equivalent:
\begin{itemize}
\item[(i)] $C$ is hyperelliptic;
\item[(ii)] $C$ carries a base point free $g_2^1$;
\item[(iii)] $t^2/h\in\oo_{P}$ for some $h\in\C[t]$ with $h(0)\neq 0$ and $\deg(h)\leq 2$;
\item[(iv)] $C$ is Gorenstein and isomorphic to a curve of degree $2g+1$ in $\mathbb{P}^{g+1}$ lying on the cone $S_{0,g}$;
\item[(v)] $C$ is Gorenstein and ${\rm gon}(C)=2$.
\end{itemize}
Here $t$ is an inhomogeneous local parameter for the normalization of $C$, centered at the preimage $\overline{P}$ of the unique singular point $P \in C$, as in Convention~\ref{cnvcnv}.
\end{thm}

\begin{proof}
The equivalence $(i) \iff (ii)$ follows immediately from Lemma ~\ref{lemmsl}.

We now prove $(ii) \iff (iii)$. From Lemma~\ref{lemgk1}, $C$ carries a base-point-free $g_2^1$ if and only if there exist polynomials $f, h \in \C[t]$ with no common factor for which $f(0)=0$, $f/h \in \op$, and
\[
2=\deg(h)+{\rm max}\{0,\deg(f)-\deg(h)\}.
\] 
In particular, if $C$ carries a base-point-free $g_2^1$, we necessarily have $\deg(f) \leq 2$, $\deg(h)\leq 2$, and moreover $h(0)\neq 0$ as $f(0)=0$ and $f,h$ have no common factor. But since $P$ is singular, $t$ is a local parameter at $\pb$, $f/h \in \op$ and $f(0)=0$, it follows that 
\begin{equation}\label{f/h}
v_{\pb}(f/h)\geq 2.
\end{equation}
On the other hand, the fact that $h(0)\neq 0$ means that $v_{\pb}(h)=0$, so from \eqref{f/h} we deduce that
\[
v_{\pb}(f)\geq 2.
\]
This in turn forces $f=at^2$ for some nonzero constant $a\in\C$, and $(iii)$ follows. Reversing each of the preceding arguments yields $(iii) \implies (ii)$.

The equivalence $(i) \iff (iv)$ may be found in \cite[Thm. 3.4.(a) $\Leftrightarrow$ (b)]{KM}; use the fact that the {\it linearly normal} curves defined there are non-Gorenstein. 

To finish the proof, note that ${\rm gon}(C)=2$ if and only if $C$ carries a $g_2^1$. If the $g_2^1$ is base-point-free, then $C$ is hyperelliptic; and, if so, $C$ is Gorenstein by \cite[Prp. 2.6.(2)]{KM}. On the other hand, if the $g_2^1$ has a base point (necessarily non-removable since $C\not\cong\pum$), then $C$ is non-Gorenstein by \cite[Cor. 2.3]{Mt}. So clearly $(ii) \iff (v)$ and we are done.
\end{proof}

Note that if $C$ is any curve of genus $g=1$, then  $\op=\C\oplus t^2\obp$; in particular, $t^2\in\op$ and so $C$ is hyperelliptic. Similarly, if $g=2$ and $\sss$ is hyperelliptic, then $\op=\C\oplus \C(t^2+at^3)\oplus t^4\obp$ for some $a\in\C$; so $t^2/(1-at)\in\op$, and hence $C$ is hyperelliptic. We next characterize the set of nonhyperelliptic curves of genus $3$ with hyperelliptic semigroups $\sss$. To this end, first recall than any curve with $2\in\sss$ is Gorenstein, since non-Gorenstein points only occur with multiplicity three or more. So whenever $2\in\sss$, the previous theorem establishes that ${\rm gon}(C)=2$ is equivalent to $C$ being hyperelliptic.

\begin{prop}\label{hypgenus3}
Every nonhyperelliptic curve $C$ of genus $3$ with $2\in\sss$ is isomorphic to a plane curve with (inhomogeneous) parametrizing functions $f_i: \mb{P}^1 \ra \mb{C}$, $0 \leq i \leq 2$ of the form
\begin{equation}\label{abcd}
f_0=1-2at+bt^2+ct^3+dt^4, \hspace{10pt} f_1=t^2-at^3, \text{ and } f_2=t^4
\end{equation}
for some $a,b,c,d\in\C$, such that $ac \neq 0$. Moreover, $C$ is trigonal.
\end{prop}

\begin{proof}
Let $\mc{L}=\mc{L}(\ww_C,H^0(\ww_C))$ denote the canonical linear series on $C$. Note that $C$ is Gorenstein, as $2\in\sss$. Since $C$ is also non-hyperelliptic by assumption, $\mathcal{L}$ defines an embedding $\varphi: C\hookrightarrow\mathbb{P}^{g-1}=\mathbb{P}^2$ and $C^{\pr}:=\varphi(C)$ is the {\it canonical model} of $C$ first described by Rosenlicht in \cite{R} and investigated more recently by Kleiman and the third author in \cite{KM}. Let $\pi: \cb\to C$ denote the normalization map; then $\varphi\circ\pi:\pum\to \mb{P}^2$ is a parametrization of the canonical model, which we denote (abusively) by $C^{\pr}=(f_0,f_1,f_2)$. 

\medskip
By general theory (see, e.g., \cite[Thm 2.8, Cor. 2.9]{St2}), 
there exists a basis of $H^0(\ww_C)$ whose generators vanish to orders $\{0,2,4\}$ in $P$.  More precisely, $\ww_C$ may be embedded in the constant sheaf of meromorphic functions so that $H^0(\ww_C)=\langle 1,x, y\rangle$ where $x:=f_1/f_0$, $y:=f_2/f_0$ are affine coordinates around the singular point $P=(1:0:0)$. Now, since $\deg\varphi\circ\pi=\deg(\ww_C)=2g-2=4$, after replacing the $f_i$, $0 \leq i \leq 2$ with appropriately-chosen linear combinations we obtain
\[
(f_0,f_1,f_2)=(1+a_1t+a_2t^2+a_3t^3+a_4t^4,t^2+at^3,t^4).
\]
(We could further assume that $f_1$ has no $t^4$-term, but this is unimportant.)

\medskip
We claim that $a_1=2a$. Indeed, let $h:=((t^2+at^3)/f_0)^2-t^4/f_0$. Thus $h \in \op$; on the other hand, we have
\[
h=(t^4+2(a-a_1)t^5+\ldots)-(t^4-a_1t^5+\ldots)=(2a-a_1)t^5+\ldots.
\]
The fact that $5\notin\sss$ now forces $2a-a_1=0$, as desired. 

\medskip
Similarly, the possibility $a_1=a_3=0$ is precluded. Otherwise, $C$ is a nonreduced (double) curve supported along
\[
\pum=(t,t^2,1+a_2t+a_4t^2)
\]
which is precluded, because (by the standing hypotheses of this section) $C$ is necessarily integral. We claim that in fact $a_1=a_3=0$ is the only situation in which $C$ admits a $g_2^1$. 
Indeed, $\op$ is locally generated by the functions $f_1/f_0$ and $f_2/f_0$. So whenever $C$ admits an inhomogeneous parametrization in $t$ with coefficients $a,b,c,d$ as in \eqref{abcd}, we have $t^2/h\not\in\op$ for any $h\in\C[t]$ with $h(0)\neq 0$ and $\deg(h)\leq 2$; it then follows from Lemma~\ref{thmhyp}.(iii) that $C$ is not hyperelliptic. An alternative path to the same conclusion is by remarking that $C$ is a 
canonical curve, and as such cannot be hyperelliptic.

\medskip
We may use Lemma \ref{lemgk1} and Theorem \ref{lemgfb} to deduce that $C$ is trigonal in some cases. For example, if $d=0$, the base-point-free $g_k^1$ induced by $f_1/f_0$ is such that $k \leq 3$ by Lemma~\ref{lemgk1}, so $k=3$ since $C$ is nonhyperelliptic. If $a=0$, then $c\neq 0$; if, moreover, $b=0$ then $C$ carries a $g_3^1$ with a non-removable base point by Theorem \ref{lemgfb}.(i). And if $f_0(a)=0$, then $t^2/(f_0/(t-a))$ yields a base-point-free $g_3^1$ by Lemma~\ref{lemgk1}. 

\medskip
A more geometric way of locating $g_3^1$'s on $C$ is by means of the subseries (of the canonical linear series) given by $\mathcal{L}(\ww_C,\langle x-x_0,y-y_0\rangle)$. Any such $\mc{L}$ corresponds to a pencil of lines passing through $Q=(1:x_0:y_0)$ and cuts out a $g_4^1$ on $C$ (since $C$ is of degree $4$) with a base point at $Q$ 
whenever $Q \in C$. If $Q$ is a smooth point of $C$, it is a removable base point of the pencil; we then obtain a $g_3^1$ defined by $\mathcal{L}(\oo_C\langle x-x_0,y-y_0\rangle,\langle x-x_0,y-y_0\rangle)$. 
Similarly, we obtain $g^1_3$'s by removing base points of pencils lying along the line at infinity; these are precisely the points $Q_0:=(0:0:1)$ (if $\deg(f_0)\leq 3$) and $Q_r=(0:f_1(r):f_2(r))$ associated to each root $r$ of $f_0(t)$. For example, in the preceding paragraph, the $g_3^1$ obtained when $d=0$ is precisely that obtained by removing $Q_0$ from the $g_4^1$ cut out by those lines passing through it. 

We conclude that $C$ is trigonal by noting that since $C$ is Gorenstein of genus $g\geq 2$, it cannot carry a $g_2^1$ with a non-removable base-point \cite[Cor. 2.3]{Mt}.
\end{proof}

\subsection{The submaximal case}

\medskip \noindent
In the study of Weierstrass points on smooth curves, bielliptic semigroups correspond to {\it bielliptic curves}; that is, curves that are 2-to-1 covers of elliptic curves. Similarly, in our setting of singular rational curves (with a unique singular point), it is natural to call a curve {\it bielliptic} whenever it comes equipped with a degree-2 morphism to an elliptic curve. Note that the elliptic target will necessarily be singular. As Theorem~\ref{biell_curves} below shows,  if the singular point of a an arbitrary bielliptic curve is nonhyperelliptic, then it is necessarily bielliptic. The converse, however, fails.

\begin{thm}\label{biell_curves}
Let  $C$ be a curve of genus $g\geq 5$ such that $P$ is nonhyperelliptic. Then the following hold:
\begin{itemize}
\item[I.] If $C$ is bielliptic, then:
\begin{itemize}
\item[(i)] $P$ is bielliptic;
\item[(ii)] $C$ is isomorphic to a curve of degree $2g+1$ in $\mathbb{P}^{g+1}$ such that: 
\begin{itemize}
\item[(a)] if $C$ is Gorenstein then it lies on a $3$-fold scroll $S_{m,n,1}$; 
\item[(b)] if $C$ is non-Gorenstein then it lies on a 4-fold scroll $S_{m,n,0,0}$
\end{itemize}
 where $m=\lfloor g/2\rfloor$ and $n=\lceil (g-4)/2\rceil$.
\item[(iii)] $C$ is at most tetragonal, and non-Gorenstein if it is trigonal.
\end{itemize}
\item [II.] If $g\geq 10$, then $C$ is bielliptic if and only if it carries a base point free $g_8^3$.
\end{itemize}
\end{thm}

\begin{proof} To prove item I.(i), assume there is a degree-$2$ morphism $\phi: C\to E$, where $E$ is elliptic. Up to isomorphism, $E$ is the projective closure of ${\rm Spec}\,[u^2,u^3]$. Let $\pb$ be the point of $\pum$ which lies over $P$ and let $t$ be the corresponding adapted local parameter. Now $\phi$ lifts to a double cover $\overline{\phi}:\pum\to\pum$ that corresponds to the inclusion $\mathbb{C}(u)\subset\mathbb{C}(t)$. Set $\overline{Q}:=\overline{\phi}(\pb)$, and let $u_{\overline{Q}}$ be the corresponding adapted local parameter. The ramification index of $\overline{\phi}$ at $\pb$ is $e_{\pb}:=\vpb(u_{\overline{Q}})$ and $\overline{\phi}$ is ramified at $\pb$ if $e_{\pb}>1$. The Riemann--Hurwitz formula implies that $\overline{\phi}$ ramifies at two points, both with simple ramification. So $\vpb(u_{\overline{Q}})\leq 2$. 

\medskip
Now assume for the sake of argument that $\phi(P)$ is nonsingular; then, by construction, $u_{\overline{Q}}\in\op$. But $P$ is nonhyperelliptic, hence $2\not\in\sss$; then $\vpb(u_{\overline{Q}})\geq 3$ is forced, which yields a contradiction. So $\phi(P)$ is necessarily the unique singular point of $E$. This means that $u_{\overline{Q}}=u$ and $u^2, u^3\in\op$. Since $2\not\in\sss$, it follows that  $\vpb(u)= 2$, hence $4, 6\in\sss$. If $3\in\sss$, we have $g\leq 3$, while if $5\in\sss$, we have $g\leq 4$; the item follows. 

\medskip
To prove I.(ii), we first apply I.(i), and deduce that $\sss$ is a hyperelliptic semigroup. As we saw in the proof of Theorem~\ref{biell}, this means that
\begin{equation}
\label{equsse}
\sss^*=
\begin{cases}
\{0,4,6,8,\ldots,2g-4,2g-3,2g-2,2g\} & \text{if}\ $\sss$\ \text{is symmetric} \\
\{0,4,6,8,\ldots,2g-2\} & \text{if}\ $\sss$\ \text{is nonsymmetric}
\end{cases}
\end{equation}
Now choose $x,y,z,u\in \op$ such that $\vpb(x)=4$, $\vpb(y)=6$, $\vpb(z)=2g-3$ (resp. $\vpb(z)=2g-1$) if $\sss$ is symmetric (resp., nonsymmetric), and $\vpb(u)=2g+1$.

\medskip
If $C$ is Gorenstein, then $\sss$ is symmetric, and the morphism
$$
\phi:=(1,x,x^2,\ldots,x^m,y,xy,\ldots,x^ny,z,xz):C\longrightarrow\mathbb{P}^{m+n+3}\ \ \ \ \ \ \ \ \ \ \ \ \
$$
is an embedding. Indeed, the morphism $\phi$ is prescribed by the linear series $\mc{L}=\mathcal{L}(\oo_C\langle V\rangle,V)$ where
\[
V=\langle 1,x,x^2,\ldots,x^m;y,xy,\ldots,x^ny;z,xz\rangle.
\]
Set $\aaa:=\oo_C\langle V\rangle$. Now $\deg(\aaa)=\vpb(xz)=2g+1$, so by \cite[Lem. 5.1.(3)]{KM} we deduce that $\aaa$ is very ample and $h^1(\aaa)=0$; it follows that $h^0(\aaa)=g+2$. So $\mathcal{L}$ embeds $C$ in $\mathbb{P}^{g+1}$ as a curve of degree $2g+1$. Note that by construction, we have $m+n+3=g+1$. Moreover, it is easy to check that
\[
{\rm rank} \begin{pmatrix}1  & x & \ldots & x^{m-1} & y & xy & \ldots & x^{n-1}y & z \\
                                    x  & x^2 & \ldots & x^{m} & xy & x^2y & \ldots & x^{n}y & xz\end{pmatrix} <2;
\] 
thus $\phi(C)$ lies on the scroll $S_{m,n,1}\subset\mathbb{P}^{g+1}$.

\medskip
Similarly, if $C$ is non-Gorenstein, then $\sss$ is nonsymmetric, and the morphism
\[
\psi:=(1,x,x^2,\ldots,x^m,y,xy,\ldots,x^ny,z,u):C\longrightarrow\mathbb{P}^{g+1}\ \ \ \ \ \ \ \ \ \ \ \ \
\]
is an embedding. Here $\deg(\psi(C))=\vpb(u)=2g+1$, and
\[
{\rm rank} \begin{pmatrix}1  & x & \ldots & x^{m-1} & y & xy & \ldots & x^{n-1}y \\
                                    x  & x^2 & \ldots & x^{m} & xy & x^2y & \ldots & x^{n}y \end{pmatrix} <2,
\]
which means that $\psi(C)$ lies on the scroll $S_{m,n,0,0}\subset\mathbb{P}^{g+1}$. 

\medskip
To prove I.(iii), since $C$ is bielliptic, composing the maps $C\to E\to\pum$ we get a 4:1 cover of the projective line, so ${\rm gon}(C)\leq 4$. Now setup $C$ as in Convention \ref{cnvcnv} and recall that any $g_k^1$ on $C$ can be computed by a sheaf of the form $\aaa:=\oo_C\langle 1,f/h\rangle$ where $f,h\in\C[t]$ have no common factor, and $\deg(\aaa)=k$. One may extend Proposition \ref{lemgk1} allowing linear series to have base points by means of the formula
$$
k=\#(\vpb(\aaa_P)\setminus\sss)+\deg(h)+{\rm max}\{0,\deg(f)-\deg(h)\}
$$
To see it, note that the first summand above agrees with $\ell(\aaa_P/\op)$. In particular, 
\begin{equation}
\label{equune}
k\geq \#(\vpb(\aaa_P)\setminus\sss)+{\rm max}\{\deg(f),\deg(h)\}
\end{equation}
Set $a:=\vpb(f/h)$ and $A:=\vpb(\aaa_P)\setminus\sss$; we may further assume $a\neq 0$ subtracting constant if necessary; besides, 
\begin{equation}
\label{equdes}
\deg(f)\geq a\ \text{if}\  a>0\ \ \ \ \ \text{and}\ \ \ \ \ \deg(h)\geq a\ \text{if}\ a<0
\end{equation}
Combining (\ref{equune}) and (\ref{equdes}), one sees that to reach gonality $3$ we may assume $|a|\leq 2$ since $m_C(P)=4$. If $|a|=1$, then $\{a,4+a,6+a\}\subset A$ and so $k\geq 4$. If $a=-2$, then $\{-2,2\}\subset A$ and so $k\geq 4$ as well. And if $a=2$ and $C$ is Gorenstein, then $\{2,2g-1\}\in A$ so $k\geq 4$ also. It follows that gonality $3$ can only be computed with $a=2$ and $C$ non-Gorenstein. In this case, if, for instance, $\op$ is monomial, then clearly $\deg(\oo_C\langle 1,t^2\rangle)=3$ and $C$ is trigonal.

\medskip
To prove item II, assume $C$ is bielliptic. Following the proof of item I.(i), let $V$ denote the vector space $\langle 1,u^2,u^3,u^4\rangle\subset k(C)$, and let $\aaa:=\oo_C\langle V\rangle$ denote the sheaf it generates. By construction, the linear series $\mc{L}=\mathcal{L}(\aaa,V)$ is three-dimensional. Because $u^2,u^3\in\op$, the sheaf $\aaa$ is invertible; and because $\aaa$ is globally-generated, the linear series $\mc{L}$ is base-point-free. Now $(1,u):\pum\to\pum$ is a double cover; it follows that $\deg(\aaa)=8$. Indeed, consider the sheaf $\G:=\oo_{\pum}\langle 1,u\rangle$ on $\pum$. Clearly $\deg(\G)=2$, so $\deg(\G^{\otimes\,4})=8$. On the other hand, by construction, $\G^{\otimes\,4}=\pi^*(\aaa)$ where $\pi:\pum\to C$ is the normalization map. Since $\aaa$ is invertible, it follows that $\deg(\aaa)=\deg(\G^{\otimes\,4})=8$.

\medskip
Conversely, assume $C$ carries a base-point-free $g_8^3$. First we claim that it is complete, i.e., that the underlying sheaf $\aaa$ is such that $h^0(\aaa)=4$. To see this, first note that $h^1(\aaa)>0$; indeed, degree considerations and Serre duality show that if $h^1(\aaa)=0$, then $g\leq 5$. Thus, by \cite[App.]{EHKS} or \cite[Lem. 3.1]{KM}, $\aaa$ satisfies the Clifford inequality 
\begin{equation}\label{clifford}
h^0(\aaa)\leq \deg(\aaa)/2+1=5.
\end{equation}
Further, equality holds in \eqref{clifford} in precisely four cases: 
\begin{itemize}
\item [(a)] $\aaa=\oo_C$, which is precluded because $\deg(\aaa)=8\neq 0$; \item [(b)] $\aaa=\ww_C$, which is precluded because $g\neq 5$; 
\item [(c)] $C$ is hyperelliptic, which is precluded by assumption; 
\item [(d)] or $C$ is such that $h^0(\oo/\mathcal{C})=1$.  
\end{itemize}
In the last case, $C$ is \emph{linearly normal} in the sense of \cite{KM}. On the other hand, if $C$ is linearly normal, then $m_P=\mathcal{C}_P$ and if $C$ carries â base point free $g_3^8$, its associated vector space $V$ is of the form $V=\langle 1,x_1,x_2,x_3\rangle$ with each $x_i\in \oo_P$. We may assume, e.g., that $x_1\in\mmp$, but then $\deg(\aaa)\geq v_{\pb}(x_1)\geq g+1$. In particular, if $g\geq 8$, we have $\deg(\aaa)\geq 9$ and so this possibility is precluded as well. Consequently, equality in \eqref{clifford} does not hold, and $\mathcal{L}$ is complete.

\medskip
Now let $\phi: C\to \mathbb{P}^3$ be the morphism induced by the $g_8^3$, and set $C^*:=\phi(C)$.
Note that $\phi$ cannot be birational; otherwise, applying Castelnuovo's genus bound \cite[p.116]{ACGH} to $\phi$ yields $g \leq 9$, contrary to assumption.
Accordingly, set $d\geq 2$ denote the degree of the finite cover $\phi$. 
Note that the sheaf $\G:=\oo_{C^*}\langle \phi_*(H^0(\aaa))\rangle$ on $C^*$ is such that $\deg(\G)=8/d$. Further, we have $h^0(\G)=4$; indeed, $\G$ may be regarded as a subsheaf of $\aaa$, so $h^0(\G) \leq h^0(\aaa)$; on the other hand, by construction, $H^0(\G)\supset \phi_*(H^0(\aaa))$ so $h^0(\G)\geq \dim(\pi_*(H^0(\aaa)))=4$. Here $h^1(\G)=0$; indeed, if $h^1(\G)>0$, the Clifford inequality would yield $h^0(\G)\leq (\deg(\G)/2)+1=(8/2d)+1\leq 3$, which cannot happen. 

\medskip
From the vanishing of $h^1(\G)$, we deduce that 
\[
h^0(\G)=\deg(\G)+1-g^*
\]
where $g^*$ is the genus of $C^*$. In other words, we have $4=(8/d)+1-g^*$, whose only possible solution is $(d=2,g^*=1)$.
That is, $\phi$ realizes $C$ as a double cover of an elliptic curve. Since $C$ is rational, $C^*$ is too. As such, $C^*$ has just one singularity which is either an ordinary cusp or an ordinary node. 
Moreover, from the proof of item I.(i) we see that $\phi$ lifts to a double cover $\overline{\phi}:\pum\to\pum$ which is ramified at $\pb$ and such that $P^*=\phi(\pb)$ is singular. Because $\overline{\phi}$ ramifies at $\pb$, $P^*$ cannot be a node;
we conclude that $C^*=E$.
\end{proof}

To produce examples of nonbielliptic curves $C$ with bielliptic singularities $P$, one possibility is to apply item II of the preceding theorem, and exhibit a $g^3_8$ that is birational (as opposed to a double cover of an elliptic curve). Another possibility is suggested by the following observation.
\begin{rem}\label{biell_remark}
Every bielliptic structure $C \ra E$ may be lifted to a double cover $\mb{P}^1 \ra \mb{P}^1$, in such a way that the following diagram commutes:
\[
\hspace{20pt}
\xymatrix{\pum \ar[d]_{\pi_C} \ar[r]^{2:1} & \pum\ar[d]^{\pi_E}.\\
         C\ar[r]^{2:1} & E}
\]
Here $\pi_C$ and $\pi_E$ are the normalizations of $C$ and $E$, respectively.
\end{rem}

\medskip
Consider, then, the following two examples of curves $C$ with locally planar bielliptic singularities $P$ (in each case, $\op$ is the localization of $\C[x,y]$ at the ideal $(x,y)$):
\begin{itemize}
\item Example 1: Let $(x,y)=(t^4,t^6+t^7)$. In this case, the unique rational curve $C$ whose unique singularity is $P$ fails to cover $E$. Indeed, note that the morphism $\phi=(1,t^4,t^6+t^7,t^8)$ associated to the $g^3_8$ of $C$ is clearly birational; it follows from item II of Theorem~\ref{biell_curves} it follows that $C$ is nonbielliptic.
\item Example 2: Let $(x,y)=(u^2,u^3)$, where $u:=\frac{t^2}{1+t^3}$. Let $C$ be the unique rational curve whose unique singularity is $P$, and let $D$ be the rational curve parametrized by $(1+t^3,t^2,t^4)$. Then $C$ covers $E$; indeed, we have the following commutative diagram:
\begin{equation}\label{D_diagram}
\ \ \ \ \ \ \ \ \ \ \ \ \
\xymatrix{\pum \ar[d]_{\pi_D} \ar[r]^{(1,u)} & \pum\ar[dd]^{\pi_E}\\
         D\ar[ur]_{(1,u)}\ar[d]    &  \\
              C \ar[r]^{(1,u^2,u^3)}                       &  E}.
\end{equation}
Proposition~\ref{hypgenus3} implies that $(1,u)$ is a morphism of degree 3, i.e., that $C \ra E$ is a triple cover. We claim that no double cover $C \ra E$ exists. Otherwise, we could write $(x,y)=(v^2,v^3)$ for some $v\in\oo_{D^{\pr},P^{\pr}}$, with $P^{\pr}$ a singular point of some $D^{\pr}$, such that $(1,v): D^{\pr} \ra \pum$ is a double cover as in \eqref{D_diagram}. However, Lemma~\ref{thmhyp} implies that any $v$ for which $(1,v)$ defines a double cover is necessarily of the form $v=t^2/h$ for some $h\in\C[t]$ with $h(0)\neq 0$ and $\deg(h)\leq 2$; we leave it to the reader to check that such a $v$ does not exist.
\end{itemize}

\end{document}